\documentclass[11pt]{article}

\usepackage[utf8]{inputenc}
\usepackage{amssymb}
\usepackage[utf8]{inputenc}
\usepackage{graphicx,amsthm,tikz-cd,amsmath,amssymb, mathtools,tikzsymbols, asymptote, suetterl, mathrsfs, color}
\usepackage[margin=0.5in]{geometry}
\usepackage{  dsfont  }
\usepackage[shortlabels]{enumitem}
\usepackage{amsfonts,datetime,fancyhdr,wasysym} 
\usepackage{amsmath}
\usepackage[pagebackref,colorlinks=true,pdfpagemode=UseNone,urlcolor=blue,linkcolor=blue,citecolor=blue]{hyperref}
\usepackage{comment}
\usepackage{cancel}

\usepackage{mathtools}
\mathtoolsset{showonlyrefs}

\newcommand\be {{\mathbf{e}}}
\newtheorem{theorem}{Theorem}[]
\newtheorem{corollary}[theorem]{Corollary}

\newtheorem{lemma}{Lemma}[section]
\newtheorem{proposition}{Proposition}[section]
\newtheorem{definition}{Definition}[section]
\newtheorem{remark}{Remark}[section]

\newtheorem{problem}{Problem}[section]

\makeatletter

\newcommand{\Rmnum}[1]{\expandafter\@slowromancap\romannumeral#1@}
\makeatother

\renewcommand\ss {{\mathfrak{s}}}
\newcommand\Rm {{\mathbb{R}}}
\newcommand\Cm {{\mathbb{C}}}
\newcommand\Zm {{\mathbb{Z}}}
\newcommand\Sm {{\mathbb{S}}}
\newcommand\Dm {{\mathbb{D}}}
\newcommand\A {{\mathcal{A}}}

\newcommand\wtM {{\widetilde M}}
\newcommand\zbar {{\bar{z}}}
\newcommand\DD {{\mathscr{D}}}
\renewcommand{\d}{{{\mathrm d}}}
\newcommand\pp {{\mathsf{p}}}

\title{A family of non-simple surfaces whose transport twistor spaces admit global blow-down maps}

\author{Fran\c{c}ois Monard\thanks{Department of Mathematics, University of California, Santa Cruz CA 95064. Email: fmonard@ucsc.edu} \and Zhengyi Qi\thanks{Department of Mathematics, University of California, Santa Cruz CA 95064. Email: zqi34@ucsc.edu}}
\date{}

\begin{document}
\maketitle
\begin{abstract}
    In the literature on X-ray transform and Transport Twistor (TT) spaces, blow-down maps (or maps with holomorphic blow-down structure as defined in \cite{localglobalblowdownstransport}) are maps that desingularize the degenerate complex structure of the TT space of an oriented Riemannian surface, while collapsing (yet separating) geodesics of the unit tangent bundle of that surface. Such maps were originally constructed in \cite{localglobalblowdownstransport} for near-constant curvature simple surfaces, showing that the interior of their TT space is biholomorphic to an open set in standard $\Cm^2$. The construction there relied on a microlocal argument leveraging the absence of conjugate points. 

    In this note, we construct an explicit example of a family of convex, non-trapping Riemannian surfaces, some of which have conjugate points, yet all of whose TT spaces admit a global blow-down map. We also discuss a consequence on the existence of special geodesically invariant functions and its application to geometric inverse problems. 
\end{abstract}


\section{Introduction}\label{sec:intro}

{\em Transport Twistor} (TT) spaces, first introduced in \cite{bohr2021transport}, have emerged as a conceptual tool to reframe in complex-geometric terms the study of geometric inverse problems and dynamical systems on the unit tangent bundle of a given Riemannian surface. Any oriented Riemannian surface $(M,g)$ with unit tangent bundle $SM \stackrel{\pi}{\longrightarrow} M$ defined by
\begin{align}
    SM = \{(x,v)\in TM,\ g_x(v,v) = 1\},
    \label{eq:SM}
\end{align}
gives rise to a TT space, i.e. an involutive manifold $(Z,\DD)$, where $Z$ can be taken to be the unit ball bundle $BM = \{(x,v)\in TM,\ g_x(v,v)\le 1\}$, and $\DD\subset T_\Cm Z$ is a distinguished involutive complex 2-plane bundle. The structure $\DD$ defines a unique almost complex structure on $Z^\circ$, while capturing geodesic transport phenomena and fiberwise holomorphicity at its boundary component $SM$, see Sections \ref{sec:prelims} and \ref{sec:TTspace} for details. 

A number of interesting features about TT spaces have been established so far. A first important property is how they encapsulate the concept of fiberwise holomorphic geodesically invariant distributions (see, e.g. \cite{bohr2021transport,bohr2024invariant}) and fiberwise holomorphic integrating factors. These objects appear as key concepts for the inversion of the attenuated geodesic X-ray transform \cite{salo2011attenuated}, the non-abelian X-ray transform and the tensor tomography problem \cite{paternain_salo_uhlmann_2023}. In \cite{bohr2024invariant}, an explicit description of algebras of holomorphic functions on the TT spaces of some closed surfaces is given. The local and global integrability of the involutive structure $\DD$ have been studied in \cite{localglobalblowdownstransport}: for simple surfaces near constant curvature ones, one may construct a map $\beta \colon Z \to \Cm^2$ satisfying $d\beta(\DD) = 0$ and providing an embedding of the complex surface $Z^\circ$ into $\Cm^2$ with minimal degeneration at the boundary component $SM$; such maps are coined ``blow-down maps" there, and a more precise definition of their defining properties (``holomorphic blow-down structure") can be found in Definition \ref{def:hbs} below. One may also define some appropriate notion of biholomorphism amongst TT spaces in \cite{bohr2025biholomorphism}, and \cite[Corollary 1.7]{localglobalblowdownstransport} provides an important identity principle to help prove that TT spaces, within the classes of simple surfaces or Anosov surfaces separately, are rather rigid, in that such biholomorphisms can only arise from ``trivial'' ones (underlying Riemannian isometries, constant rescaling of the metric, or the antipodal map $(x,v)\mapsto (x,-v)$ on $BM$), see \cite[Theorems 1, 2]{bohr2025biholomorphism}.

This note is concerned with the construction of blow-down maps in cases of {\em non-simple}\footnote{A Riemannian surface $(M,g)$ with boundary is {\em simple} if $\partial M$ is geodesically convex, $M$ is non-trapping and $M$ has no conjugate points.} surfaces, specifically, convex and non-trapping surfaces {\em with conjugate points}, showing that embeddability $Z^\circ\hookrightarrow \Cm^2$ (and, more specifically, having holomorphic blow-down structure) can still be achieved in non-simple cases; see Theorem \ref{thm:main} below. The current proof is an explicit construction which is out of the reach of the machinery set forth in \cite{localglobalblowdownstransport}, one of whose key steps leverages simplicity. More specifically, the construction in \cite{localglobalblowdownstransport} produces an ansatz for a map $\beta$ which requires invertibility of the ``normal operators" $N_0$ and $N_1$ of the geodesic X-ray transform on functions and on sections of $T^{1,0}M$, respectively. This invertibility crucially requires that the underlying surface be simple, notably as the absence of conjugate points is what makes $N_0,N_1$ elliptic pseudo-differential operators on the interior of $(M,g)$. 

In the literature on two-dimensional geometric inverse problems on manifolds with boundary, simplicity is an intriguing interface between positive and negative results. Simple surfaces are known to be boundary distance rigid \cite{pestov2005two}, where the geodesic X-ray transform is injective \cite{Mukhometov1975} and H\"older-stable (see, e.g., \cite[Theorem 4.6.4]{paternain_salo_uhlmann_2023}). When conjugate points enter the picture, H\"older stability positively breaks down \cite{monard2015geodesic}, yet injectivity holds in rotation-invariant situations \cite{sharafutdinov1997integral} or over piecewise constant functions \cite{ilmavirta2020geodesic}, and is still open in general convex, non-trapping surfaces. On the dual side, the construction of explicit preimages for an adjoint of the X-ray transform (the backprojection operator), or equivalently, the construction of geodesically invariant distributions with prescribed fiberwise Fourier modes, is of interest to tensor tomography and generally a cornerstone of this literature, see e.g. \cite[Chapter 8]{paternain_salo_uhlmann_2023} or \cite{pestov2005two}. Such existence/surjectivity results largely leverage simplicity, and our main result appears to positively push them past one of the simplicity thresholds. We state a corollary to that effect below (see Corollary \ref{cor}).

We now state the main result and give an outline of the remainder at the end of the next section. 

\section{Statement of the main result}

\subsection{Preliminaries} \label{sec:prelims}

Here and below, $(M,g)$ denotes an oriented Riemannian surface with unit tangent bundle $SM$ as in \eqref{eq:SM}, and $X\in C^\infty(SM; T(SM))$ will denote the generator of the geodesic flow on $SM$. We denote $Z = BM$ the unit ball bundle of $(M,g)$. We equip $Z$ with a rank-2 involutive distribution $\DD \subset T_\Cm Z$ making $(Z,\DD)$ an involutive manifold (with corners when $\partial M\ne \emptyset$). The distribution $\DD$, addressed more explicitly in Section \ref{sec:TTspace} below, is characterized (uniquely so, by virtue of \cite[Theorem 2.6]{Bohr2024}) by the following properties: 
\begin{enumerate}
    \item[(a)] $\DD \cap \overline{\DD} = \begin{cases} \Cm X, & SM, \\ 0, & Z\backslash SM. \end{cases}$ \\
	    Away from $SM$, this implies that $\DD = \ker (J+i)$ for a complex structure $J$. In particular, $(Z^\circ,J)$ is a classical complex surface.
    \item[(b)] For every $x\in M$, the fiber $Z_x = B_x M$, equipped with its standard complex structure induced by $g$ and the orientation, is holomorphically embedded in $Z$.
    \item[(c)] The orientation induced by $J$ is the canonical one on $TM$. 	
\end{enumerate}
We say that a map $\beta\colon Z\to \Cm_{w,\xi}^2$ is {\bf holomorphic} if $d\beta|_{(x,v)} (\DD|_{(x,v)}) \subset (\Cm \partial_{\bar{w}} \oplus \Cm \partial_{\bar{\xi}})|_{\beta(x,v)}$ for all $(x,v)\in Z$, and we are interested in the existence of such maps that are maximally injective, and how this question relates to the geometry of $(M,g)$. By condition (a), the components of a holomorphic map must be geodesically invariant on $SM$ (i.e., in $\ker X$), and as such, $\beta|_{SM}$ can, at best, separate geodesics, but not points on them. On the other hand, on $Z^\circ$ where the Newlander-Nirenberg theorem guarantees local biholomorphisms (see, e.g. \cite[Theorem 2.6.19]{Huybrechts2005}, injectivity can potentially be obtained globally, and this is indeed the case for the TT space of simple Riemannian surfaces of near-constant curvature, as proved in \cite{localglobalblowdownstransport}. There, to make such maps stable under metric perturbation, the notion of map with {\bf holomorphic blow-down structure} was defined. While we postpone a full definition to Section \ref{sec:proofmain} (see Definition \ref{def:hbs}), our main result, Theorem \ref{thm:main} below, is an existence statement for such maps for a family of non-simple surfaces, which in particular implies: 

(1) The restriction $\beta|_{SM}$ separates geodesics.

(2) The restriction $\beta|_{Z^\circ}$ is a biholomorphism onto its image. 

\subsection{Main result}

Consider, for $R>0$, the surface $M = \Dm_R := \{(x,y)\in \Rm^2, x^2+y^2\le R^2\}$, equipped with the metric $g_\kappa$, expressed in polar coordinates $(x,y)=P(r,\theta) := (r\cos\theta,r\sin\theta)$, 
\begin{align}
    P^* g_\kappa = (1+\kappa r^2)^2 \ \d r^2 + r^2 \d\theta^2.
    \label{eq:g}
\end{align}
The metric $g_\kappa$ is indeed smooth at the origin $(x,y)=(0,0)$, and with Gaussian curvature $K(r)= \frac{2\kappa}{(1+\kappa r^2)^3}$. An oriented orthonormal frame for $P^* g_\kappa$ is given by $(\be_r,\be_\theta) := \left( \frac{1}{1+\kappa r^2} \partial_r, \frac{1}{r}\partial_\theta\right)$. After covering the construction of the TT space in Section \ref{sec:TTspace}, we show in Lemma \ref{lem:be} that 
\begin{align}
    \be := P_* (\cos \theta \be_r - \sin\theta\be_\theta)
    \label{eq:be}
\end{align}
is a smooth section of $SM$, and as such induces a global chart on $Z=BM$ given by 
\begin{align} 
    M\times \Dm \ni ((x,y),\nu) \mapsto ((x,y), \nu\cdot \be) \in BM,
    \label{eq:Zcoords}
\end{align}
where "$\cdot$" denotes the complex multiplication defined in \eqref{eq:complexmult}, induced by the metric and orientation of $(M,g)$. Below, we write $z= x+iy = re^{i\theta}$ for computational convenience, although this is, in general, not a holomorphic coordinate on $(\Dm_R,g_\kappa)$. In the coordinates $(z,\nu)$, the involutive distribution of $Z$ takes the form $\DD = \Cm\partial_{\bar{\nu}} \oplus \Cm \Xi$, where 
\begin{align}
\Xi := \left(  (2 + \kappa z\bar{z})  \nu^2 -\kappa z^2\right) \partial_z + \left( 2 + \kappa z \bar{z} -\kappa\nu^2\bar{z}^2\right) \partial_{\bar{z}} + \kappa \left( z - \nu^2 \bar{z} \right) \left( \bar{\nu} \partial_{\bar{\nu}} - \nu \partial_\nu \right).
\label{eq:Xi}
\end{align}

Our main result is as follows. 
\begin{theorem}\label{thm:main}
Consider $(M,g) = (\Dm_R, g_\kappa)$ defined by \eqref{eq:g} where $R>0$ and $\kappa>-\frac{1}{R^2}$. On its transport twistor space $(Z,\DD)$, in the coordinates $(z,\nu)$ defined in \eqref{eq:Zcoords}, define
\begin{align}
    w(z,\nu) := (z-\bar{z}\nu^2)e^{\frac{\kappa(z\bar{z}-\bar{z}^2\nu^2)}{2}},\qquad \xi(z,\nu) := \nu e^{\frac{\kappa(z\bar{z}-\bar{z}^2\nu^2)}{2}}, \qquad (z,\nu)\in \Dm_R\times \Dm.
\end{align}
Then the map $\beta := (w,\xi)$ has holomorphic blow-down structure. 
\end{theorem}

As a result, $Z^\circ$ embeds into $\Cm^2$ for all such cases of $(R,\kappa)$. This is the first example of an explicit blow-down map in non-constant curvature, as well as in non-simple cases. Indeed, we also establish in Section \ref{sec:characsimple} below that $(\Dm_R, g_\kappa)$ has conjugate points whenever $\kappa R^2 \ge 1$, see also Corollary \ref{cor:gkappa} and an illustration of it on Figure \ref{fig:gkappa}.

\begin{figure}[htpb]
    \centering
    \includegraphics[clip, trim = 0 0 0 17, width=0.18\linewidth]{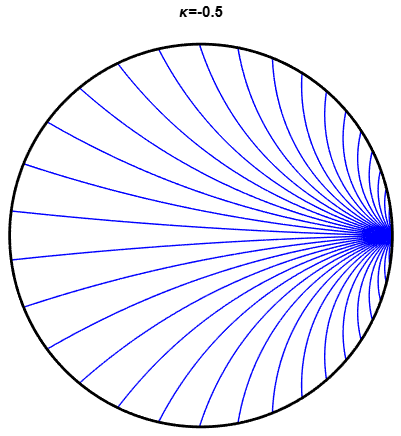}
    \includegraphics[clip, trim = 0 0 0 17, width=0.18\linewidth]{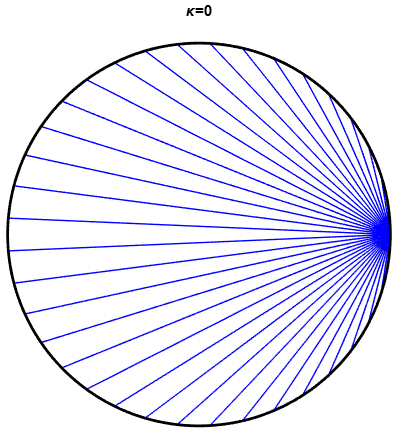}
    \includegraphics[clip, trim = 0 0 0 17, width=0.18\linewidth]{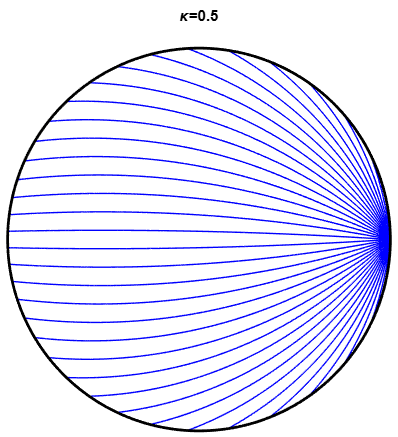}
    \includegraphics[clip, trim = 0 0 0 17, width=0.18\linewidth]{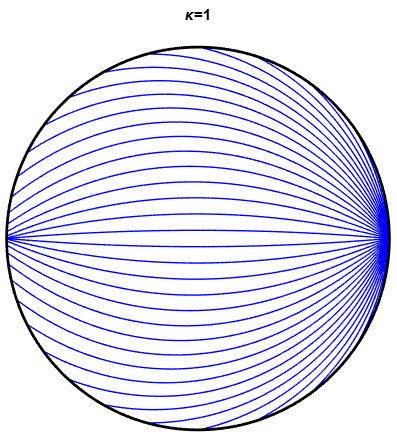}
    \includegraphics[clip, trim = 0 0 0 17, width=0.18\linewidth]{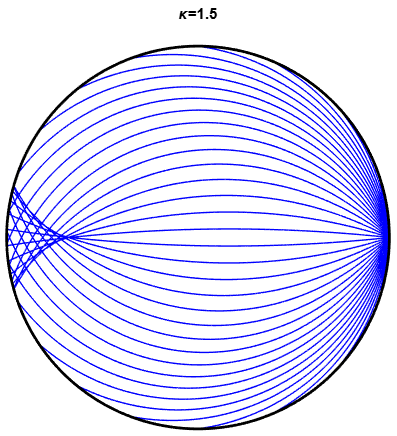}
    \caption{Left to right: some geodesics of $(\Dm_R, g_\kappa)$ defined in \eqref{eq:g} for $R=1$ and $\kappa \in \{-0.5, 0, 0.5, 1, 1.5\}$. Conjugate points appear at the threshold value $\kappa =1$.}
    \label{fig:gkappa}
\end{figure}

\subsection{Consequences for inverse problems: invariant distributions}

The metric and orientation of $(M,g)$ define a unique circle action $\rho_t\colon SM\to SM$ (given by  \eqref{eq:complexmult} below with $\omega = e^{it}$) with infinitesimal generator $V$. Defining 
\begin{align}
    \Omega_k := \{f\in C^\infty (SM) \colon Vf = ikf\}, \quad k\in \Zm,
\end{align}
any $u\in C^\infty(SM)$ has a unique decomposition $u = \sum_{k\in\Zm} u_k$ into its vertical Fourier modes $u_k\in\Omega_k$. Of special interest for integral-geometric inverse problems are  functions $u\in C^\infty(SM)$ which are geodesically invariant (i.e., satisfying $Xu=0$ on $SM$) and with a Fourier support condition of the form $u_k = 0$ for all $k< m$ for some $m\ge 0$ (if $m=0$, this corresponds to the definition of being {\em fiberwise holomorphic}, as originally defined in \cite{salo2011attenuated}). When this is the case, the "bottom" mode $u_m$ satisfies $\eta_- u_m = 0$, where $\eta_- := \frac{1}{2} (X-i[X,V])$ is the so-called Guillemin-Kazhdan operator, elliptic on sections of $\Omega_m$. Hence an important problem in general is: 
\begin{problem}\label{prob}
Fix $(M,g)$ an oriented Riemannian surface. Given $m\ge 0$ and $\phi\in \Omega_m\cap \ker\eta_-$, does there exist $u\in C^\infty (SM) \cap \ker X$ with $u_m = \phi$ and $u_k = 0$ for $k<m$?
\end{problem}
 When $(M,g)$ is a simple surface, a positive answer to Problem \ref{prob} is a consequence of \cite[Lemma 5.6]{Paternain2013a}, based on microlocal arguments which leverage simplicity. Other explicit constructions in the Euclidean setting are given in \cite{Monard2017a}, and their relevance to tensor field reconstructions from their geodesic X-ray transform can be seen in, e.g., \cite{Monard2017a,assylbekov2018inversion}.

Problem \ref{prob} can be reinterpreted in Transport Twistor terms as follows. See also \cite[Sec. 2.1 and 3.1.1]{bohr2024invariant} for a more extensive discussion. By \cite[Proposition 4.4]{bohr2021transport}, the space of smooth, fiberwise holomorphic, (geodesically) invariant functions on $SM$ is in 1:1 correspondence with the algebra of "holomorphic" functions on TT space
\begin{align}
    \A(Z) = \{u\in C^\infty(Z),\ \d u (\DD) = 0\},
\end{align}
with the isomorphism given by the map $\A(Z)\ni f\mapsto f|_{SM}$. The same map induces, for any $m\ge 0$, an isomorphism between invariant smooth functions on $SM$ of the form $u = \sum_{k\ge m} u_k$ and 
\begin{align}
    \A_m(Z) := \Big\{f\in \A(Z)\colon f|_{SM} = \sum_{k\ge m} u_k \Big\},
\end{align}
and the associated ``bottom mode" $u_m \in \Omega_m \cap \ker \eta_-$ is in 1:1 correspondence with an element of the space of holomorphic differentials
\begin{align}
    {\cal H}_m & := \{a\in C^\infty (M,\otimes^m (T^{1,0}M)^*)\colon \bar\partial a= 0\},
\end{align}
via the map ${\cal H}_m \ni a \mapsto ( (x,v) \mapsto a_x (v^{\otimes m}))$. Given an element $f\in \A_m(Z)$, we may view the $m$-th Fourier mode $(f|_{SM})_m$ as an element $a\in {\cal H}_m$ which we denote $a = {\pi_m}_* f$. This scenario is summarized in the following (not necessarily exact) sequence 
\begin{align}
        0\to \A_{m+1}(Z) \hookrightarrow \A_m(Z) \stackrel{ {\pi_m}_*}\longrightarrow {\cal H}_m \to 0.
        \label{eq:sequence}
\end{align}
 Then Problem \ref{prob} is equivalent to asking whether \eqref{eq:sequence} is exact, which is known to be true for simple surfaces, in particular in the case where $|\kappa|R^2 <1$ in the model $(\Dm_R, g_\kappa)$ above. 

 In the non-simple cases of this family, we now prove a slightly weaker result as follows. A holomorphic coordinate for $(\Dm_R, g_\kappa)$ defined above is given by $\zeta = z e^{\kappa z\zbar/2}\in \Dm_{R_1}$, where $R_1 = R e^{\kappa R^2/2}$. Then an element $a\in {\cal H}_m$ takes the form $A(\zeta)\d\zeta^m$, where $\partial_{\bar\zeta} A = 0$ on $\Dm_{R_1}^\circ$ and $A\in C^\infty(\Dm_{R_1})$. Such elements admit a power series in $\zeta$
\begin{align}
    A(\zeta) = \sum_{k=0}^\infty A_k \zeta^k, \qquad A_k \in \Cm,
    \label{eq:pse}
\end{align}
with radius of convergence at least $R_1$. Below, for $R_2>R_1$, we say that an element $a\in {\cal H}_m$ {\em extends smoothly to $\Dm_{R_2}$} if the radius of convergence of \eqref{eq:pse} is at least $R_2$ and the power series extends smoothly to $\Dm_{R_2}$.

\begin{corollary}[Holomorphic extensions]\label{cor}
    Consider $(M,g) = (\Dm_R, g_\kappa)$ defined by \eqref{eq:g} where $R>0$ and $\kappa \ge 0 $, and define $R_1 := R e^{\kappa R^2/2}$ and $R_2 = 2R e^{\kappa R^2}>R_1$. Fix $m\ge 0$. If $a \in {\cal H}_m$ extends smoothly to $\Dm_{R_2}$, then there exists $f\in \A_m(Z)$ such that $\pi_{m*} f = a$. Equivalently, for $\phi$ the unique element of $\in \Omega_m\cap \ker\eta_-$ associated with $a$, there exists $u\in C^\infty (SM)$ such that $Xu=0$, $u_m = \phi$ and $u_k = 0$ for all $k<m$.
\end{corollary}

One might conjecture that the sequence \eqref{eq:sequence} is still exact in those cases, although this result is beyond the reach of the current proof, which also does not recover exactness of the sequence in the simple range $\kappa R^2 \in [0,1)$. 

\paragraph{Outline.} The remainder of the article is organized as follows. Since $(\Dm_R,g_\kappa)$ is an example of disk with a rotationally invariant metric satisfying Herglotz' nontrapping condition (in short, a {\em Herglotz disk}), we first discuss generalities about the family of Herglotz disks in Section \ref{sec:model}, including a characterization of simplicity expressed in terms of a {\em scattering function} $\ss(\alpha)$, see Proposition \ref{prop:characsimplicity}. We then discuss the TT space in Section \ref{sec:TTspace}, including the coordinate-free construction, and coordinate systems that make computations tractable. We end with the proofs of Theorem \ref{thm:main} and Corollary \ref{cor} in Section \ref{sec:proofmain}.

\section{The model} \label{sec:model}

\subsection{Convex, nontrapping surfaces} \label{sec:definitions}

Given an oriented Riemannian surface $(M,g)$ with boundary $\partial M$, we denote the geodesic flow $\varphi_t\colon SM\to SM$, defined for 
\begin{align*}
    \{(x,v,t),\ (x,v)\in SM,\ t\in [-\tau(x,-v),\tau(x,v)]\},
\end{align*}
where $\tau(x,v)\in [0,\infty]$ is the first exit time of the geodesic passing through $(x,v)$ at $t=0$. The boundary $\partial SM$ is the union of the incoming/outgoing ($+/-$) boundaries
\begin{align}
\partial_\pm SM = \{(x,v)\in SM,\ x\in \partial M,\ \pm g_x(v,\nu_x)\ge 0\},
\label{eq:in_out}    
\end{align}
where $\nu_x$ is the inward-pointing unit normal at $x\in \partial M$.

Here and below, we will assume that (i) $\partial M$ is {\bf strictly convex} in the sense of having positive definite second fundamental form, and that (ii) $(M,g)$ is {\bf non-trapping} in the sense that $\sup_{(x,v)\in SM}\tau(x,v)<\infty$. 

 When $(M,g)$ is convex and non-trapping, the manifold-with-boundary $\partial_+ SM$ is used as a model for all oriented geodesics through $M$, and the {\bf scattering relation} $S: \partial SM \to \partial SM$, defined by
\begin{align}
    S(x,v) := 
    \begin{cases}
    \varphi_{\tau(x,v)}(x,v), &  (x,v)\in \partial_+ SM, \\
    \varphi_{-\tau(x,-v)}(x,v), &  (x,v)\in \partial_- SM,
    \end{cases}
    \label{eq:scatrel}
\end{align}
is a smooth diffeomorphism of $\partial SM$ satisfying $S\circ S = \mathrm{Id}$, with set of fixed points the glancing boundary $\partial_0 SM = \partial_+ SM\cap \partial_- SM$.

\subsection{Rotation-invariance}

A special family of convex, non-trapping surfaces are the topological disks equipped with a rotation-invariant metric of the form 
\begin{align}
    P^* g = a^2(r)\ \d r^2 + r^2 \d\theta^2.
    \label{eq:ametric}
\end{align}
 Smoothness of the metric at the origin implies that 
\begin{align}
    a(r) = 1 + r^2 \tilde{a}(r^2), \qquad \text{for some function} \quad \tilde{a}\in C^\infty ([0,R^2]).
    \label{eq:tildea}
\end{align}
Indeed, if $(x,y)=(r\cos\theta,r\sin\theta)$, the function $g(\partial_x,\partial_x) = 1+\cos^2\theta (a(r)^2-1)$ can only be smooth at the origin if $a(0) = 1$ and the Taylor expansion of $a(r)^2-1$ off $r=0$ is only made of even terms. The latter implies a similar conclusion on $a$.

In the form \eqref{eq:ametric}, such metrics automatically satisfy Herglotz' non-trapping condition (see, e.g., \cite[Def. 2.3.3]{paternain_salo_uhlmann_2023}), which implies both convexity and non-trapping. Below, we call such surfaces {\em Herglotz disks}. 

To study the scattering relation on a Herglotz disk, we introduce the \emph{fan-beam coordinates} $(\theta,\alpha)$ on $\partial SM$, where $\theta$ parameterizes a boundary point $x = Re^{i\theta} \in \partial \Dm_R$, and $\alpha$ parametrizes the tangent vector $v(\alpha) = -\cos\alpha\ \be_r - \sin\alpha\ \be_\theta$. We thus denote for short $\tau(\alpha):= \tau (Re^{i\theta}, v(\alpha))$, a quantity independent of $\theta$. In such coordinates, the symmetries induced by rotation invariance imply the reduced form for the scattering relation
\begin{align}
    S(\theta,\alpha) = (\theta'(\theta,\alpha), \alpha'(\theta,\alpha)) = (\theta + \pi + 2\ss(\alpha), \pi - \alpha), \qquad (\theta,\alpha) \in (\Rm/2\pi\Zm)^2,    
    \label{eq:S}
\end{align}
where $\ss(\alpha)$ satisfies
\[
\ss(\alpha+\pi) = \pi + \ss(\alpha),
\]
a consequence of the involutive property $S \circ S = \mathrm{Id}$. Reflection symmetry also implies that $\ss(-\alpha) = -\ss(\alpha)$, in particular $\ss(0) = 0$, and since $S$ fixes points on $\partial_0 SM$, we have $\ss(\pm \pi/2) = \pm \pi/2$. For example, $\ss(\alpha) = \alpha$ for a Euclidean disk ($a\equiv 1$). The scattering function $\ss(\alpha)$ is the only part of the scattering relation which may help set apart two rotation-invariant metrics. 

Our main result for this section is the following: 
\begin{lemma}\label{lem:ssgkappa}
    The scattering function of the model $(\Dm_R,g_\kappa)$ is given by
    \begin{align}
	\ss(\alpha) = \alpha- \frac{\kappa R^2}{2} \sin(2\alpha).
	\label{eq:ssgkappa}
    \end{align}    
\end{lemma}

\begin{proof}[Proof of Lemma \ref{lem:ssgkappa}]
Using Hamiltonian formalism, the geodesic vector field on $SM$ takes the expression
\begin{align}
    X = \frac{\cos\alpha}{1 + \kappa r^2} \partial_r + \frac{\sin\alpha}{r} \partial_\theta - \frac{\sin\alpha}{r(1 + \kappa r^2)} \partial_\alpha,
    \label{eq:X}    
\end{align}
and Clairaut's first integral is $r\sin\alpha$ (i.e., $X(r\sin\alpha) =0$). From the symmetries of $\ss$ discussed above, it is enough to prove \eqref{eq:ssgkappa} for $\alpha\in (-\pi/2,0)$. The latter situation amounts to considering a geodesic arc $(r(t),\theta(t))$ at which $r(0) = R$ and where $\dot{\theta}(0)>0$. By virtue of \cite[Theorem 2.3.4]{paternain_salo_uhlmann_2023}, such an arc consists of two branches, one where $r(t)$ decreases from $R$ to $\rho = -R\sin\alpha$, the other where $r(t)$ increases from $\rho$ to $R$. For such geodesics, if $(\rho,\theta_{\mathrm{v}})$ are the polar coordinates of the vertex point (where $r$ reaches its minimum), \cite[Eq. (2.30)]{paternain_salo_uhlmann_2023} gives the equation (with $b(r) = r$ there) 
\begin{align*}
    \theta(t) = \theta_{\mathrm{v}} \mp \rho \int_\rho^{r(t)} \frac{a(r)}{r} \frac{\d r}{\sqrt{r^2 - \rho^2}},
\end{align*}
where $+/-$ refers to the ascending/descending branch. At the endpoints where $r=R$, we also have $\theta(\tau) = \theta(0) + \pi + 2\ss(\alpha)$. Combining this with the last display, and particularizing to $a(r) = 1+\kappa r^2$, this gives the expression 
\begin{align*}
    \ss(\alpha) = \frac{\theta(\tau)-\theta(0)}{2} - \frac{\pi}{2} = -R \sin \alpha \int_{-R \sin \alpha}^{R} \frac{1 + \kappa r^2}{r \sqrt{r^2 - R^2 \sin^2 \alpha}}\, \d r - \frac{\pi}{2}.
\end{align*}
To compute the last expression, we change variable $u = \sqrt{r^2-R^2\sin^2\alpha}$ to make appear
\begin{align*}
    \ss(\alpha) &= -R\sin\alpha \int_0^{R\cos\alpha} \frac{1+\kappa (u^2+R^2\sin^2\alpha)}{u^2+R^2\sin^2\alpha}\ \d u - \frac{\pi}{2} \\
    &= -R\sin\alpha \left[ \frac{1}{R\sin\alpha} \tan^{-1} \left( \frac{u}{R\sin\alpha} \right) +\kappa   \right]_{u=0}^{u=R\cos\alpha} - \frac{\pi}{2} \\
    &= -\tan^{-1} \cot \alpha - \frac{\pi}{2} - \frac{\kappa R^2}{2} \sin(2\alpha) \\
    &= \alpha - \frac{\kappa R^2}{2} \sin(2\alpha),
\end{align*} 
and the result follows.
\end{proof}

\subsection{A characterization of simplicity for Herglotz disks} \label{sec:characsimple}

A convex, non-trapping Riemannian surface is called {\bf simple} if and only if it has no conjugate points. A pair  of points $p,q\in M$ is conjugate along the geodesic $\gamma (t)$ whenever there exists $t_0<t_1$ such that $p = \gamma(t_0)$, $q = \gamma(t_1)$ and there exists a nontrivial solution $j(t)$ of the scalar Jacobi equation 
\begin{align}
    \ddot j(t) + K (\gamma(t)) j(t) = 0, \quad j(t_0)=j(t_1) = 0, \qquad \left(\dot j \equiv \frac{dj}{dt}\right)
    \label{eq:jacobi}
\end{align}
where $K$ is the Gaussian curvature of $(M,g)$. Our aim is to show that the simplicity of Herglotz disks can be tested directly from the scattering function $\ss(\alpha)$. 

\begin{proposition}\label{prop:characsimplicity}
Let \( (\Dm_R,g) \) be a Herglotz disk with $g$ of the form \eqref{eq:ametric}. Then \( (\Dm_R,g) \) is simple if and only if the scattering function $\ss(\alpha)$ satisfies \( \ss'(\alpha) >0 \) for all $\alpha\in [-\pi/2,\pi/2]$.
\end{proposition}

To prove Proposition \ref{prop:characsimplicity}, we introduce the Jacobi function $b$ defined on 
\begin{align}
    \{(\alpha,t), \quad \alpha\in [-\pi/2,\pi/2],\ t\in [0,\tau(\alpha)] \},
\end{align}
unique solution to 
\begin{align}
    \ddot{b} + K(\gamma_{Re^{i\theta}, v(\alpha)}) b = 0, \quad b(\alpha,0) = 0, \quad \dot b(\alpha,0) = 1,
\end{align}
a quantity independent of $\theta$. By convexity of $\partial M$, $\lim_{\alpha\to \pm \pi/2} \tau(\alpha) = 0$, and for all $\alpha$ such that $\tau(\alpha)$ less than the radius of injectivity of $(M,g)$, one must have $b(\alpha,t)>0$ for all $t\in (0,\tau(\alpha)]$.

We now explain how to express $\ss'(\alpha)$ in terms of the function $b$.
\begin{lemma}
Let $g$ be the form (\ref{eq:ametric}) on the disk $\mathbb{D}_{R}$. Then the scattering function satisfies:
\begin{align}
    \ss'(\alpha)=\frac{b(\alpha,\tau(\alpha))}{2R\cos\alpha}.
    \label{eq:sprime}
\end{align}
\end{lemma}
\begin{proof}
    Recalling that $\mu(x,v) = g_x(v,\nu_x)$ for $(x,v)\in \partial SM$, a framing of $T(\partial SM)$ is given by $\{V,H:= (V\mu)X + \mu [X,V]\}$. In fan-beam coordinates $(\theta,\alpha)$ where $\mu = \cos\alpha$, this reads 
    \begin{align*}
        V = \partial_\alpha, \qquad H =-\sin\alpha X+\cos\alpha [X,V] =-\frac{1}{R}\partial_{\theta}+\frac{1}{R a(R)}\partial_\alpha.
    \end{align*}
    From \eqref{eq:S}, it follows that 
    \begin{align*} 
        dS(V) = -2\ss'(\alpha)\partial_\theta+\partial_\alpha.
    \end{align*}
    On the other hand, \cite[Lemma A.1]{monard2023sampling} gives
    \begin{align*}
        dS(V)=\frac{-b(x,v,t)}{\mu(x,v)\circ S} H+\dot{b}(
    x,v,t)V, \qquad (x,v)\in \partial SM.        
    \end{align*}
    In the coordinates $(\theta,\alpha)$, this gives
    \begin{align*}
        dS(V)=\frac{-b(\alpha,\tau(\alpha))}{R\cos\alpha}\partial_\theta+\left(\frac{b(\alpha,\tau(\alpha))}{R a(R) \cos\alpha}+ \dot{b}(\alpha,\tau(\alpha))\right)\partial_\alpha.
    \end{align*}
    Comparing the $\partial_\theta$ components of both expressions for $dS(V)$, the result follows. 
\end{proof}

We are now ready to prove Proposition \ref{prop:characsimplicity}.

\begin{proof}[Proof of Proposition \ref{prop:characsimplicity}]
\textbf{(\(\implies\))} Suppose there exists \( \alpha_0 \) such that \( \ss'(\alpha_0) \leq 0 \).  
Since \( \ss'(\alpha) \) is continuous in \( \alpha \) and we know that \( b(\alpha, \tau(\alpha)) > 0 \) for \( \alpha \) close to \( \pm\frac{\pi}{2} \), by the intermediate value theorem, there must exist some \( \alpha_0 \in (-\frac{\pi}{2}, \frac{\pi}{2}) \) such that $\ss'(\alpha_0) =0$, or equivalently, $b(\alpha_0, \tau(\alpha_0)) = 0$, and hence $p=R e^{i 0}$ and $q = R e^{i(\pi + 2\ss(\alpha_0))}$ are conjugate. 

\medskip

\textbf{(\(\impliedby\))} Suppose there exists a pair of conjugate points $(p,q)$ in $M$. By rotation invariance, we may choose $p,q$ to be conjugate along a geodesic $\gamma(t)$ such that $\gamma(0) = R e^{i 0}$, $\gamma(t_0) = p$ and $\gamma(t_1) = q$ for some $t_0,t_1$ satisfying $0\le t_0<t_1$, $\dot{\gamma}(0) = v(\alpha_0)$ for $\alpha_0\in (-\pi/2,\pi/2)$. Then there is a nontrivial solution $j(t)$ of \eqref{eq:jacobi} vanishing at $t_0,t_1$. If $t_0 =0$, then $b(t_1,\alpha_0) = 0$ and we set $t^* = t_1$. If $t_0\ne 0$, since $j(t)$ and $b(t,\alpha_0)$ both solve the same Jacobi equation, by Sturm's Separation Theorem, there exists $t^* \in (t_0,t_1)$ such that $b(\alpha_0,t^*) = 0$. 

At this point, we have proved that there exists $\alpha_0 \in (-\pi/2,\pi/2)$ and $0<t^* \le \tau(\alpha_0)$ such that $b(\alpha_0,t^*)=0$. Now if $t^* = \tau(\alpha_0)$, then $\ss'(\alpha_0) =0$ and we are done. If $t^*< \tau(\alpha_0)$, we now analyze the structure of the zero set of $b(\alpha,t)$ to prove that there is $\alpha_1$ such that $b(\alpha_1,\tau(\alpha_1)) =0$ so that $\ss'(\alpha_1) = 0$. Since $b(\alpha,\cdot)$ is a non-trivial solution of a second-order linear homogeneous ODE in $t$, one must have $\partial b/\partial t \ne 0$ whenever $b(\alpha,t) =0$. Hence by the implicit function theorem, there is a function $t(\alpha)$ defined in a neighborhood of $\alpha_0$ such that $t(\alpha_0) = t^*$ and $b(\alpha, t(\alpha)) = 0$. One must also satisfy that $t(\alpha)$ must be greater than the radius of injectivity of $(M,g)$. Since $\lim_{\alpha\to \pm \pi/2} \tau(\alpha) = 0$, the maximal extension of $t(\alpha)$ must encounter a point of the form $(\alpha_1,\tau(\alpha_1))$ with $\tau(\alpha_1) >0$. Hence we have found $\alpha_1$ with $\tau(\alpha_1)>0$ such that $b(\alpha_1,\tau(\alpha_1))=0$ and hence $\ss'(\alpha_1) =0$. Proposition \ref{prop:characsimplicity} is proved. 
\end{proof}

Combining Proposition \ref{prop:characsimplicity} with \eqref{eq:ssgkappa}, we deduce: 

\begin{corollary}\label{cor:gkappa} Fix $R>0$.

If $|\kappa|R^2<1$, then $(\Dm_R,g_\kappa)$ defined in \eqref{eq:g} is simple.

If $\kappa \ge \frac{1}{R^2}$, then $(\Dm_R,g_\kappa)$ has conjugate points. 
\end{corollary}

\begin{proof}
    Equation \eqref{eq:ssgkappa} implies that $\ss'(\alpha) = 1 - \kappa R^2 \cos(2\alpha)$, hence the condition ``$\ss'(\alpha)>0$ for all $\alpha\in [-\pi/2,\pi/2]$" needed in Proposition \ref{prop:characsimplicity} is satisfied if and only if $|\kappa|R^2<1$.
\end{proof}

\section{Transport Twistor space and the distribution $\DD$} \label{sec:TTspace}

\subsection{Coordinate-invariant construction of $Z$}

We now recall the construction of the TT space of an oriented Riemannian surface $(M,g)$. 

The metric and orientation give rise to a unique almost-complex structure, i.e. a bundle map $\iota\colon TM\to TM$ satisfying $\iota^2 = -id$. For $v\in TM$, $\iota v$ is the unique tangent vector such that $g(\iota v,\iota v) = g(v,v)$, $g(\iota v, v) = 0$ and $(v,\iota v)$ is oriented. The almost-complex structure allows us to define a complex multiplication 
\begin{align}
    (x,\omega\cdot v) := (x,\text{Re}(\omega)v + \text{Im}(\omega) \iota v), \qquad (x,v)\in TM, \quad \omega\in \Cm,
    \label{eq:complexmult}
\end{align}
which also gives rise to the usual circle action on the fibers of $SM$ when $|\omega|=1$. Recalling that $Z=BM$, a way to construct the involutive structure $\DD\subset T_\Cm Z$ in a coordinate-invariant way, is as the pushdown $p_\star \widehat{\DD}$ of a distribution defined on the total space $SM\times \Dm$ of the principal $\Sm^1$-bundle 
\begin{align}
    \pp\colon SM \times \Dm \to BM, \qquad (x,v,\omega)\mapsto (x,\omega\cdot v),
    \label{eq:p}
\end{align}
with fibers given by the orbits of the circle action
\begin{align}
    \psi_t (x,v,\omega) = (x, e^{it}\cdot v, e^{-it}\omega), \quad (x,v)\in SM, \quad t\in \Rm/2\pi\Zm.
    \label{eq:circleZ}
\end{align}
If $X$ denotes the generator of the geodesic flow on $SM$, and $V$ denotes the generator of the circle action on the fibers of $SM$, we recall the Guillemin-Kazhdan operators $\eta_\pm := \frac{1}{2}(X\pm i [X,V])\in C^\infty(M; T_\Cm (SM))$. Then the distribution $\widehat{\DD}$ is given by 
\begin{align}
    \widehat{\DD} = \text{span} \{ \eta_- + \omega^2 \eta_+,\ \partial_{\overline{\omega}}  \}.
\end{align}
It is then proved in \cite[Lemma 4.1]{bohr2021transport} that $\widehat{\DD}$ is $\Sm^1$-invariant and as such descends to a distribution $\DD = \pp^* \widehat{\DD}$ on $Z$, satisfying the properties (a)-(b)-(c) of Section \ref{sec:prelims}.

\subsection{The distribution in polar coordinates}

The polar coordinate chart is a diffeomorphism $\wtM := (0,R]_r \times (\Rm/2\pi\Zm)_\theta \to M\backslash\{0\}$. On $(\widetilde{M}, P^* g_\kappa)$, an oriented orthonormal frame of $T \wtM$ is given by $(\be_r, \be_\theta)$, where $\be_r = \frac{1}{1+\kappa r^2} \partial_r$ and $\be_\theta = \frac{1}{r}\be_\theta$, and the circle action applied to $\be_r$ gives rise to the global chart 
\begin{align*}
    \wtM\times \Rm/2\pi\Zm\ni (r,\theta,\alpha) \mapsto (r,\theta, \cos\alpha\ \be_r + \sin \alpha\ \be_\theta) \in S\wtM,
\end{align*}
where $V = \partial_\alpha$. In these coordinates, the almost-complex structure of $M$ is given by 
\begin{align}
    \Cm (\be_r + i\be_\theta) = \Cm \Big(\partial_{\bar{z}} - \frac{\kappa z^2}{2+\kappa z\zbar}\partial_z\Big) \qquad (z = re^{i\theta}),
    \label{eq:acstruct}
\end{align}
and the geodesic vector field takes the expression \eqref{eq:X}, from which we can derive an expression for the Guillemin-Kazhdan operators: 
\begin{align*}
    \eta_\pm = \frac{1}{2}(X \pm i [X,V]) = \frac{e^{\pm i\alpha}}{2} \left( \be_r \mp i \be_\theta \pm \frac{i}{r(1+\kappa r^2) }\partial_\alpha\right).
\end{align*}
The distribution $\widehat{\DD}$ on $S\wtM_{(r,\theta,\alpha)}\times \Dm_\omega$ is thus given by $\widehat{\DD} = \text{span} \{ \partial_{\bar{\omega}}, \widehat{\Xi}\}$, where
\begin{align*}
    \widehat{\Xi} = \eta_- + \omega^2 \eta_+ = \frac{1}{2} \left( (e^{-i\alpha} + \omega^2 e^{i\alpha}) \be_r + i (e^{-i\alpha}-\omega^2 e^{i\alpha}) \be_\theta + \frac{i(-e^{-i\alpha}+\omega^2 e^{i\alpha})}{r(1+\kappa r^2)} \partial_\alpha   \right).
\end{align*}
To push $\widehat{\DD}$ down to $B\wtM$, we use the chart
\begin{align}
\widetilde{\Phi}\colon \wtM\times \Dm \ni (r,\theta,\mu) \mapsto (r,\theta, \text{Re}(\mu) \be_r + \text{Im}(\mu) \be_\theta),
\label{eq:chart2}
\end{align}
and notice that $\widetilde{\Phi}(r,\theta,\mu) = \pp (r,\theta,0,\mu)$ where $\pp$ is defined in \eqref{eq:p}. In particular, we have that $\pp_* \partial_r = \partial_r$, $\pp_* \partial_\theta = \partial_\theta$, $\pp_* (\partial_\omega) = \partial_\mu$ and $\pp_* (\partial_{\bar{\omega}}) = \partial_{\bar{\mu}}$ above a point $(r,\theta,0,\omega)$. Finally, to compute $\pp_* \partial_\alpha$, we use that $\pp_*$ annihilates the infinitesimal generator of \eqref{eq:circleZ} given by $\partial_\alpha - i (\omega\partial_\omega- \bar{\omega}\partial_{\bar{\omega}})$, hence $\pp_*(\partial_\alpha) = \pp_* (i (\omega\partial_\omega- \bar{\omega}\partial_{\bar{\omega}})) = i (\mu\partial_\mu- \bar{\mu}\partial_{\bar{\mu}})$.

As a conclusion, the involutive structure of $B\wtM$ looks like $\widetilde{\DD} = \Cm\partial_{\bar\mu} \oplus \Cm\widetilde{\Xi}$, where
\begin{align}
\widetilde{\Xi} := \pp_* (\widehat{\Xi}|_{(r,\theta,0,\mu)}) = \frac{1}{2}\left(\frac{(\mu^2 + 1)}{1+\kappa r^2}\partial_r + \frac{i(1 - \mu^2)}{r} \partial_\theta + \frac{1 - \mu^2}{r(1+\kappa r^2)} \left( \mu \partial_\mu - \bar{\mu} \partial_{\bar{\mu}} \right)\right).
\label{eq:Dpolar}
\end{align}

\subsection{Smooth coordinates on $Z$ down to the origin}\label{sec:coordZ}

The discussion in this section applies to any disk $\Dm_R$ equipped with a rotation-invariant metric of the form \eqref{eq:ametric} for some positive function $a(r)$ satisfying \eqref{eq:tildea}. We denote by $(\be_r,\be_\theta) = (\frac{1}{a}\partial_r,\frac{1}{r}\partial_\theta)$ an oriented orthonormal frame of $P^* g$. Although $(P_* \be_r,P_* \be_\theta)$ does not extend to the origin of $M$, we have the following. 
\begin{lemma}\label{lem:be}
    The section 
    \begin{align*}
        \be := P_* (\cos \theta \be_r - \sin\theta\be_\theta),
    \end{align*}
    originally defined and smooth on $M\backslash \{0\}$, extends to a smooth section of $SM$.
\end{lemma}

\begin{proof}[Proof of Lemma \ref{lem:be}]
    For $(x,y)\ne (0,0)$, we compute directly that 
    \begin{align*}
        \be &= \frac{1}{a} \left( (1 + \cos^2\theta (1-a)) \partial_x + \cos\theta\sin\theta (1-a) \partial_y\right) \\
        &\! \stackrel{\eqref{eq:tildea}}{=} \frac{1}{a(x^2+y^2)}\left( 1 - x^2 \tilde{a} (x^2+y^2)) \partial_x - xy \tilde{a}(x^2+y^2) \partial_y\right),
    \end{align*}
    hence the result. 
\end{proof}

Lemma \ref{lem:be} justifies the validity of the chart defined in \eqref{eq:Zcoords}. Moreover, relative to the chart \eqref{eq:chart2}, this corresponds to the change of variable 
\begin{align}
    (r,\theta,\mu,\bar{\mu}) \to (z,\bar{z},\nu,\bar\nu), \quad \text{where} \quad z = r e^{i\theta}, \quad \nu = \mu e^{i\theta}.
    \label{eq:znuchart}
\end{align}
Pushing forward the distribution $\widetilde{\DD} = \Cm \partial_{\bar\mu}\oplus \Cm \widetilde{\Xi}$, we find that $\partial_{\bar{\mu}} = e^{-i\theta} \partial_{\bar{\nu}} \propto \partial_{\bar{\nu}}$, and a straightforward calculation shows that the vector field $\widetilde{\Xi}$ defined in \eqref{eq:Dpolar} becomes $\frac{e^{-i\theta}}{2(1+\kappa r^2)} \Xi$, where $\Xi$ is given in \eqref{eq:Xi}.

\section{Proofs of Theorem \ref{thm:main} and Corollary \ref{cor}} \label{sec:proofmain}

Before proving Theorem \ref{thm:main}, we first recall the definition of {\em holomorphic blow-down structure}, originally defined in \cite{bohr2024invariant} as a way to make conditions (1)-(2) of Section \ref{sec:prelims} open, and hence amenable to perturbation theory. First, to make (1) quantitative, one must account for a systematic degeneracy at the glancing $\partial_0 SM$, which is embodied by the $C_\alpha^\infty$-structure on $\partial_+ SM$ defined in \cite[Section 2.3]{bohr2024invariant}. Since checking condition (a) below boils down to an explicit computation outlined in the proof of Theorem \ref{thm:main}, we skip the specific definition. Second, as points of $Z\backslash SM$ approach $SM$, the blow-down behavior of a holomorphic map $\beta\colon Z\to \Cm^2$ satisfying (1)-(2) can be captured using a special Hermitian metric $\underline{\Omega}$ on $Z$, and the standard $\Omega_{\mathbb{C}^2}:=i(\d w\wedge \d\bar{w}+\d\xi\wedge \d\bar{\xi})$ on $\Cm^2_{w,\xi}$. We define $\underline{\Omega}$ in terms of the coordinates $(z,\nu)$ in \eqref{eq:znuchart} (see also \cite[Sec. 2.4.1]{bohr2024invariant} for a coordinate invariant description): write $\DD = \Cm \partial_{\bar{\nu}}\oplus \Cm \Xi$ with $\Xi$ defined in \eqref{eq:Xi}, then 
\begin{align}
    \underline{\Omega} := i(\eta_1\wedge\bar{\eta}_1+\eta_2\wedge\bar{\eta}_2),
\end{align}
where $\eta_1,\eta_2$ are the unique one-forms on $Z$ satisfying
\begin{align}
\begin{split}
    \eta_1 (\partial_\nu) &= \eta_1 (\partial_{\bar{\nu}}) = \eta_1(\Xi) = 0,\qquad \eta_1(\overline{\Xi}) = 1-|\nu|^4, \\
    \eta_2 (\overline{\Xi}) &= \eta_2 (\partial_{\bar{\nu}}) = \eta_2(\Xi) = 0,\qquad \eta_2(\partial_\nu) = 1.
\end{split}    
\label{eq:eta}
\end{align}

\begin{definition}\label{def:hbs}
    A smooth map $\beta\colon Z\to \mathbb{C}^2$ has {\bf holomorphic blow-down structure} if it has the following properties:
    \begin{itemize}
        \item[(a)] The restriction $\beta|_{\partial_{+}SM}:\partial_{+}SM\rightarrow\mathbb{C}^2$ is a totally real $C^{\infty}_{\alpha}-$embedding.
        \item[(b)] The restriction $\beta|_{Z^\circ}$ is a biholomorphism onto its image.
        \item[(c)] Writing $\beta^*\Omega_{\Cm^2} = i \sum_{j,k=1}^2 H_{jk} \eta_j\wedge \bar{\eta}_k$ for some Hermitian matrix field $H = \{H_{jk}\}_{j,k=1}^2$, there exists $c>0$ such that $H-cId$ is positive semidefinite\footnote{In \cite{localglobalblowdownstransport} such a condition is written concisely as $\beta^*\Omega_{\Cm^2} \geq c\underline{\Omega}$.}.
    \end{itemize}
\end{definition}

\begin{proof}[Proof of Theorem \ref{thm:main}]
\textbf{Proof of (a).} Following \cite{localglobalblowdownstransport}, in order to show the totally real $C^{\infty}_{\alpha}-$embedding, we need first show that $\beta|_{\partial_{+} SM}$ is injective. Evaluating $\beta$ at $(z,\nu)=(Re^{i\theta},e^{i(\theta+\pi+\alpha)})$ with $\alpha\in[\frac{-\pi}{2},\frac{\pi}{2}]$, we write
\begin{align}
    (w,\xi)(Re^{i\theta},e^{i(\theta+\pi+\alpha)})= ((1-e^{2i\alpha})Re^{i\theta}e^{\frac{\kappa R^2(1-e^{2i\alpha})}{2}}, -e^{i\alpha}e^{i\theta}e^{\frac{\kappa R^2(1-e^{2i\alpha})}{2}}).
\end{align}
Notice that $w/\xi = 2iR\sin\alpha$ uniquely determines $\alpha\in [-\pi/2,\pi/2]$, after which $e^{i\theta}$ is uniquely determined by either $w$ or $\xi$. By \cite[Lemma 4.1]{localglobalblowdownstransport}, the $C^{\infty}_{\alpha}$-embedding property boils down to verifying that the rescaled Jacobian 
\[
\frac{1}{\cos^2(\alpha)} \left| \frac{\partial w}{\partial\theta} \frac{\partial\xi}{\partial\alpha} - \frac{\partial w}{\partial\alpha} \frac{\partial \xi}{\partial\theta} \right|^2
\]
is nowhere-vanishing. We thus compute
\begin{align*}
\left[
    \begin{array}{cc}
       \frac{\partial w}{\partial\theta}  & \frac{\partial w}{\partial\alpha} \\
        \frac{\partial \xi}{\partial\theta} & \frac{\partial \xi}{\partial\alpha} 
    \end{array}\right] = i e^{i\theta} e^{\frac{\kappa R^2}{2}(1 - e^{2i\alpha})} \left[
    \begin{array}{cc}
       R (1 - e^{2i\alpha}) & R e^{2i\alpha} (-2 - \kappa R^2+ \kappa R^2e^{2i\alpha}) \\
       -e^{i\alpha} & -e^{i\alpha} (1 - \kappa R^2 e^{2i\alpha})
    \end{array}\right].
\end{align*}

Substituting, we find:
\[
\frac{1}{\cos^2(\alpha)} \left| \frac{\partial w}{\partial\theta} \frac{\partial\xi}{\partial\alpha} - \frac{\partial w}{\partial\alpha} \frac{\partial \xi}{\partial\theta} \right|^2 = 4 R^2|e^{\kappa R^2(1 - e^{2i\alpha})}| \neq 0.
\]

\textbf{Proof of (b).}  By \cite{localglobalblowdownstransport}, it suffices to prove that the Cauchy-Riemann equations are satisfied and that $\beta$ is bijective onto its image. By the inverse function theorem, injectivity and holomorphicity will show that the inverse (defined on $\beta(\Dm_R^\circ\times \Dm^\circ)$) will be automatically holomorphic. The computation that $\d w (\DD) = \d\xi (\DD) = 0$ is a direct calculation which we omit, best done in the polar chart \eqref{eq:chart2}, where 
\begin{align}
    w = r e^{i\theta} (1-\mu^2) e^{\kappa r^2(1-\mu^2)/2}, \qquad \xi = \mu e^{i\theta} e^{\kappa r^2 (1-\mu^2)/2},    
    \label{eq:wxiPolar}
\end{align}
and $\DD$ is given in \eqref{eq:Dpolar}. 

Next, we check injectivity on $Z^\circ$, by recovering {\em at most} one pair $(z,\nu)\in \Dm_R\times \Dm^\circ$ from the relations
\begin{align*}
w = (z-\nu^2 \bar{z}) e^{\kappa(z\zbar - \nu^2 \zbar^2)/2}, \qquad \xi = \nu e^{\kappa(z\zbar - \nu^2 \zbar^2)/2}.
\end{align*}
First observe that if $\xi = 0$, then necessarily $\nu = 0$ and we are left with the relation $w = z e^{\kappa z\zbar/2}$. Writing $u := |z|^2$, we derive $|w|^2 = u e^{\kappa u} =: f(u)$. Noticing that $f'(u) = (1+\kappa u) e^{\kappa u}$, the hypothesis $1+\kappa R^2>0$ implies that $f'(u)>0$ on $[0,R^2]$. Hence equation $f(u)= |w|^2$ determines a unique $u = |z|^2$, then $z$ is determined via $z = w e^{-\kappa |z|^2}$.

Now suppose $\xi\ne 0$. Then $\nu \ne 0$, and we can define the quantity 
\begin{align}
    c := \frac{w}{\xi} = \frac{z-\nu^2 \zbar}{\nu} = r\frac{1-\mu^2}{\mu}.
    \label{eq:c}
\end{align}
With $c$ known from data, the above relation helps recover $z$ when $\nu$ is determined, via
\begin{align}
    z = \frac{\nu}{1 - |\nu|^4}(c + \bar{c}|\nu|^2).
    \label{eq:znu}
\end{align}
Substituting \eqref{eq:znu} into $\xi$ and $|\xi|^2$, we obtain
\begin{align}
    \xi &= \nu e^{\frac{\kappa |\nu|^2}{2(1-|\nu|^4)} (|c|^2 + |\nu|^2 c^2)}, \label{eq:nutoxi}\\
    |\xi|^2 &= h_c(|\nu|^2), \qquad h_c(u):= u e^{\frac{\kappa u}{1-u^2} (|c|^2 + u \text{Re}(c^2))}. \label{eq:nutomodxi}
\end{align}
Writing $c = c_R+ic_I$ with $c_R,c_I\in \Rm$, notice that $h_c(u) = ue^{\kappa u \big( \frac{c_R^2}{1-u} + \frac{c_I^2}{1+u}  \big)}$. If equation \eqref{eq:nutomodxi} determines a unique $|\nu|^2\in [0,1)$, then $\nu$ is uniquely determined from \eqref{eq:nutoxi}. Hence we now study the solvability of \eqref{eq:nutomodxi}. Differentiating, we find
\begin{align}
    h'_c(u) = \left( 1 + \kappa u \left( \frac{c_R^2}{(1-u)^2} + \frac{c_I^2}{(1+u)^2}     \right) \right) \exp\left( \kappa u \left( \frac{c_R^2}{1-u} + \frac{c_I^2}{1+u}  \right)  \right).    
    \label{eq:hprimec}
\end{align}
In addition, observe that for any solution $u = |\nu|^2$ of \eqref{eq:nutomodxi}, the reconstructed $z$ from \eqref{eq:znu} satisfies 
\begin{align*}
    z\zbar = u \left( \frac{c_R^2}{(1-u)^2} + \frac{c_I^2}{(1+u)^2}\right)
\end{align*}
Comparing to \eqref{eq:hprimec}, we thus find that $\text{sgn} (h'_c(u)) = \text{sgn}(1+\kappa z\zbar)$. Since we seek solutions satisfying $|z|\le R$ with $\kappa R^2>-1$, we must have $1+\kappa z\zbar>0$, and thus any admissible solution $u$ to \eqref{eq:nutomodxi} must also satisfy $h'_c(u)>0$. We will conclude by showing that the region $\{h'_c >0\}$ consists of at most a contiguous interval of the form $[0,u_c)\subseteq [0,1)$, and where $h_c$ is injective so that \eqref{eq:nutomodxi} can have at most one admissible solution. To prove this last statement: if $\kappa\ge 0$, it's immediately clear from \eqref{eq:hprimec} that $h'_c(u)>0$ on $[0,1)$; if $\kappa <0$, notice that for all $u\in [0,1)$
\begin{align*}
    \frac{\d}{\d u} \left(  1 + \kappa u \left( \frac{c_R^2}{(1-u)^2} + \frac{c_I^2}{(1+u)^2}     \right)  \right) = \kappa \left( c_R^2 \frac{1+u}{(1-u)^3} + c_I^2 \frac{1-u}{(1+u)^3}\right).
\end{align*}
If $c=0$, then $h_c(u) = u$ is clearly injective. If $c\ne 0$, then the right side in the above display is negative on $[0,1)$, and since $h'_c(0) = 1>0$, the sign of $h'_c(u)$ can only go from positive to negative once as $u$ increases. The conclusion is reached. 

\medskip
\textbf{Proof of (c).} If $(\Xi^\vee, \overline{\Xi}^\vee, \partial_\nu^\vee, \partial_{\bar{\nu}}^\vee)$ are the dual one-forms to $(\Xi,\overline{\Xi}, \partial_\nu, \partial_{\bar{\nu}})$ on $Z^\circ$, we have 
\begin{align*}
    \beta^* \d f = \overline{\Xi}f \ \overline{\Xi}^\vee + \cancel{\Xi f}\ \Xi^\vee + \partial_\nu f\ \partial_\nu^\vee + \cancel{\partial_{\bar{\nu}} f}\ \partial_{\bar{\nu}}^\vee = \frac{\overline{\Xi}f}{1-|\nu|^4}\ \eta_1 + \partial_{\nu} f\ \eta_2, \qquad f\in \{w,\xi\},
\end{align*}
where $\eta_1,\eta_2$ are defined in \eqref{eq:eta}. Denoting $A:=e^{\kappa (z\zbar - \zbar^2 \nu^2)/2}$, a computation gives
\begin{align}
\begin{split}
    \beta^* \d w &=  A(a\eta_1 + b\eta_2), \qquad \beta^* \d\xi = A(c\eta_1 + d\eta_2), \quad \text{where} \\
    \left[
    \begin{array}{cc}
       a & b \\ c & d
    \end{array}\right] &:= \left[
    \begin{array}{cc}
       2(1+\kappa\bar{z}z-\kappa\nu^2\bar{z}^2) & -2\nu \bar{z} - \kappa\nu\bar{z}^2z +\kappa\nu^3\bar{z}^3 \\
       2\kappa \bar{z} \nu & 1 - \kappa \bar{z}^2 \nu^2
    \end{array}\right].
\end{split}   
\label{eq:abcd}
\end{align}
Thus the matrix $H$ from Definition \ref{def:hbs} has components 
\begin{align*}
    H_{11}=|A|^2(|a|^2+|c|^2),\quad H_{12}=\overline{H}_{21}=|A|^2(a\bar{b}+c\bar{d}), \quad H_{22}=|A|^2(|d|^2+|b|^2).
\end{align*}
To estimate the minimal eigenvalue $\lambda_{\min}$ of \( H \), following \cite[Theorem 4.1]{localglobalblowdownstransport}, 
\[
\lambda_{\min} \geq \frac{\det(H)}{\text{tr}(H)}=|A|^2\frac{|ad-bc|^2}{|a|^2+|b|^2+|c|^2+|d|^2}.
\]
When $1+\kappa R^2>0$, we will have that $|A|^2\geq e^{-2|\kappa|R^2}$.
Given that $\det(H) = 2+2\kappa z\bar{z}$, then $\det (H)\geq 2-2|\kappa|R^2$, whenever $\kappa R^2 \in (-1,1)$. When $\kappa R^2\geq 1$, we have $2+2\kappa z\bar{z}\geq 2$. In either case, $\det (H)$ is bounded below by some positive constant. To bound $\text{tr}(H)$ from above, we find uniform bounds:
\begin{align*}
    |a|\leq 2(1+2|\kappa|R^2), \quad|b|\leq 2R(1 +|\kappa|R^2),\quad|c|\leq2|\kappa|R,\quad |d|\leq 1+|\kappa| R^2.
\end{align*}
Combining these estimates gives a uniform positive lower bound for $\lambda_{\min}$, and fulfills (c).
\end{proof}

\begin{remark} Upon restricting the relation $\d\beta(\DD)=0$ to $SM$, it follows that $(w|_{SM},\xi|_{SM})$ are smooth and fiberwise holomorphic invariant functions given by (setting $\mu = e^{i\alpha}$ in \eqref{eq:wxiPolar})
\begin{align}
    w|_{SM} = re^{i\theta}(1-e^{2i\alpha})e^{\frac{\kappa r^2(1-e^{2i\alpha})}{2}}, \quad \xi|_{SM} = e^{i\alpha}e^{i\theta}e^{\frac{\kappa r^2(1-e^{2i\alpha)}}{2}}.
    \label{eq:wxiSM}
\end{align}
That $X(\xi|_{SM})=X(w|_{SM})=0$ can also be checked directly using \eqref{eq:X}, while fiberwise holomorphicity can be directly inferred from their expression in terms of $e^{i\alpha}$.
\end{remark}

\begin{proof}[Proof of Corollary \ref{cor}]
    Recall that in the chart $(z,\nu)$ given in \eqref{eq:znuchart}, the function $\zeta = z e^{\kappa z\zbar/2}$ is a holomorphic coordinate in $(\Dm_R,g_\kappa)$, with a smooth inverse $z(\zeta)$ (note that the map $[0,R]\ni r\mapsto r e^{\kappa r^2/2}$ is smoothly invertible precisely when $\kappa R^2 >-1$). Take $a\in {\cal H}_m$ as in the statement. Then $a$ takes the form $A(\zeta)\d\zeta^m$, where $A$ has a power series \eqref{eq:pse} that extends smoothly to $\Dm_{R_2}$. Then with the functions $(w,\xi)$ defined in Theorem \ref{thm:main}, first observe that for $\kappa\ge 0$, $\max_Z |w| = R_2$, and thus the expression $A\circ w$ makes sense in $Z$ and is smooth there.  Define the function $f(z,\nu) := \xi(z,\nu)^m A(w(z,\nu))$. It is smooth on $Z$ and satisfies $\Xi f = \partial_{\bar{\nu}}f = 0$ on $Z$ as a direct consequence of Leibniz' rule. Moreover, $f|_{SM} = (\xi|_{SM})^m A(w|_{SM})$ is of the form $\sum_{k\ge m} u_k$ as can be directly inspected from the expression \eqref{eq:wxiSM}, and where the lowest Fourier mode in $e^{i\alpha}$ is $u_m = (e^{i\alpha} e^{i\theta} e^{\kappa r^2/2})^m A(r e^{i\theta} e^{\kappa r^2/2})$. The claim follows upon observing that $\zeta = re^{i\theta} e^{\kappa r^2/2}$ and 
    \begin{align*}
        e^{i\alpha} e^{i\theta} e^{\kappa r^2/2} = e^{i\theta} e^{\kappa r^2/2} ((1+\kappa r^2) \d r + i r\d \theta) (\cos\alpha\ \be_r + \sin\alpha\ \be_\theta) = \d\zeta (e^{i\alpha} \cdot \be_r),
    \end{align*}
    that is, $u_m(r,\theta,\alpha) = A(r e^{i\theta} e^{\kappa r^2/2}) \d\zeta^m( (e^{i\alpha}\cdot \be_r)^{\otimes m})$. Hence $\pi_{m*} f = a$
\end{proof}

\paragraph{Acknowledgement.} F.M. acknowledges partial funding from NSF CAREER grant DMS-1943580.

\bibliographystyle{plain}

\begin{thebibliography}{10}

\bibitem{assylbekov2018inversion}
Yernat~M Assylbekov, Fran{\c{c}}ois Monard, and Gunther Uhlmann.
\newblock Inversion formulas and range characterizations for the attenuated
  geodesic ray transform.
\newblock {\em Journal de Math{\'e}matiques Pures et Appliqu{\'e}es},
  111:161--190, 2018.

\bibitem{Bohr2024}
Jan Bohr.
\newblock Lecture notes: Transport twistor spaces in inverse problems.
\newblock 2024.

\bibitem{bohr2024invariant}
Jan Bohr, Thibault Lefeuvre, and Gabriel~P Paternain.
\newblock Invariant distributions and the transport twistor space of closed
  surfaces.
\newblock {\em Journal of the London Mathematical Society}, 109(5):e12894,
  2024.

\bibitem{bohr2025biholomorphism}
Jan Bohr, Fran{\c{c}}ois Monard, and Gabriel~P Paternain.
\newblock Biholomorphism rigidity for transport twistor spaces.
\newblock {\em Proceedings of the Royal Society A}, 481(2312):20240809, 2025.

\bibitem{localglobalblowdownstransport}
Jan Bohr, François Monard, and Gabriel~P. Paternain.
\newblock Local and global blow-downs of transport twistor space.
\newblock {\em arXiv preprint 2403.05985}, 2024.

\bibitem{bohr2021transport}
Jan Bohr and Gabriel~P Paternain.
\newblock The transport {O}ka-{G}rauert principle for simple surfaces.
\newblock {\em Journal de l’\'Ecole polytechnique — Math\'ematiques},
  10:727--769, 2023.

\bibitem{Huybrechts2005}
Daniel Huybrechts.
\newblock {\em Complex geometry: an introduction}.
\newblock Springer, 2005.

\bibitem{ilmavirta2020geodesic}
Joonas Ilmavirta, Jere Lehtonen, and Mikko Salo.
\newblock Geodesic x-ray tomography for piecewise constant functions on
  nontrapping manifolds.
\newblock In {\em Mathematical Proceedings of the Cambridge Philosophical
  Society}, volume 168, pages 29--41. Cambridge University Press, 2020.

\bibitem{Monard2017a}
Fran\c{c}ois Monard.
\newblock Efficient tensor tomography in fan-beam coordinates. {II}: attenuated
  transforms.
\newblock {\em Inverse Problems and Imaging}, 12(2):433--460, 2018.
\newblock 

\bibitem{monard2023sampling}
Fran{\c{c}}ois Monard and Plamen Stefanov.
\newblock Sampling the {X}-ray transform on simple surfaces.
\newblock {\em SIAM Journal on Mathematical Analysis}, 55(3):1707--1736, 2023.

\bibitem{monard2015geodesic}
Fran{\c{c}}ois Monard, Plamen Stefanov, and Gunther Uhlmann.
\newblock The geodesic ray transform on riemannian surfaces with conjugate
  points.
\newblock {\em Communications in Mathematical Physics}, 337(3):1491--1513,
  2015.

\bibitem{Mukhometov1975}
R.~Mukhometov.
\newblock Inverse kinematic problem of seismic on the plane.
\newblock {\em Math. Problems of Geophysics, Akad. Nauk. SSSR, Sibirsk. Otdel.,
  Vychisl. Tsentr, Novosibirsk}, 6(2):243--252, 1975.

\bibitem{Paternain2013a}
Gabriel~P Paternain, Mikko Salo, and Gunther Uhlmann.
\newblock On the range of the attenuated ray transform for unitary connections.
\newblock {\em International Mathematics Research Notices}, 2015(4):873--897,
  2015.

\bibitem{paternain_salo_uhlmann_2023}
Gabriel~P. Paternain, Mikko Salo, and Gunther Uhlmann.
\newblock {\em Geometric Inverse Problems: With Emphasis on Two Dimensions}.
\newblock Cambridge Studies in Advanced Mathematics. Cambridge University
  Press, 2023.

\bibitem{pestov2005two}
Leonid Pestov and Gunther Uhlmann.
\newblock Two dimensional compact simple riemannian manifolds are boundary
  distance rigid.
\newblock {\em Annals of mathematics}, pages 1093--1110, 2005.

\bibitem{salo2011attenuated}
Mikko Salo and Gunther Uhlmann.
\newblock The attenuated ray transform on simple surfaces.
\newblock {\em Journal of Differential Geometry}, 88(1):161--187, 2011.

\bibitem{sharafutdinov1997integral}
Vladimir~A Sharafutdinov.
\newblock Integral geometry of a tensor field on a surface of revolution.
\newblock {\em Siberian Mathematical Journal}, 38(3):603, 1997.

\end{thebibliography}

\end{document}